\documentclass{article}

\usepackage{amsthm,amsfonts,amssymb,amsmath}
\usepackage{a4wide}
\usepackage{cite}

\newtheorem{theorem}{Theorem}

\newtheorem{definition}{Definition}
\newtheorem{example}{Example}
\newtheorem{remark}{Remark}
\newtheorem{problem}{Problem}

\begin{document}

\title{Fractional Isoperimetric Noether's Theorem\\
in the Riemann--Liouville Sense\footnote{Submitted 12-Oct-2012;
revised 05-Jan-2013; accepted 23-Jan-2013; for publication 
in \emph{Reports on Mathematical Physics}.}}

%--------------------------------------

\author{Gast\~{a}o S. F. Frederico\footnote{On leave from
Department of Science and Technology,
University of Cape Verde, Praia, Santiago, Cape Verde.
Email: gastao.frederico@docente.unicv.edu.cv}\\
{\tt gastao.frederico@ua.pt}
\and Delfim F. M. Torres\\
{\tt delfim@ua.pt}}

\date{CIDMA -- Center for Research and Development in Mathematics and Applications,\\
Department of Mathematics, University of Aveiro, 3810-193 Aveiro, Portugal}

%--------------------------------------

\maketitle

%--------------------------------------

\begin{abstract}
We prove Noether-type theorems for fractional isoperimetric variational problems
with Riemann--Liouville derivatives. Both Lagrangian and
Hamiltonian formulations are obtained. Illustrative examples,
in the fractional context of the calculus of variations,
are discussed.

\bigskip


% \ccode{2010 Mathematics Subject Classification: 26A33, 49K05, 49S05.}
\end{abstract}

%--------------------------------------

\section{Introduction}

During the last fifteen years, the fractional calculus of variations
and fractional mechanics have increasingly attracted
the attention of many researchers --- see, e.g.,
\cite{MR2421931,CD:FredericoTorres:2006,Rev:04,Rev:05,Rev:06,MyID:206,MyID:221,MyID:225}
and references therein. For the state of the art,
we refer to the recent book \cite{book:frac}.

One of the oldest and interesting class of variational problems are the
isoperimetric problems \cite{CD:Bruce:2004}. Isoperimetry in
mathematical physics has roots in the Queen Dido problem
of the calculus of variations, and has recently been subject
to several investigations in the context of fractional calculus
\cite{CD:RiRuDe:2011,MyID:217,MyID:207,MyID:226}.
Here we prove Noether-like theorems for fractional isoperimetric problems
of the calculus of variations, both in Lagrangian (Theorem~\ref{theo:TNfRL})
and Hamiltonian (Theorem~\ref{thm:mainResult:FDA06}) forms.

Noether's universal principle establishes a relation between the existence
of symmetries and the existence of conservation laws, and is
one of the most beautiful results of the calculus of variations and mechanics
\cite{MyID:174,Torres:2004} and optimal control
\cite{Rev:03,CD:JMS:Torres:2002a,MR1901565}.
Noether's principle has been proved as a theorem
in various contexts \cite{MyID:106,MyID:028}.
What is important to remark here is that Noetherian
conservation laws appear naturally in closed systems,
and that in practical terms such systems do not exist:
forces that do not store energy, so-called non-conservative or
dissipative forces, are always present in real systems.
In presence of external non-conservative
forces, Noether's theorem and respective conservation laws cease
to be valid. However, it is still possible to obtain a Noether-type
theorem which covers both conservative (closed system) and
non-conservative cases.
Roughly speaking, one can prove that Noether's conservation laws are
still valid if a new term, involving the non-conservative forces, is
added to the standard constants of motion \cite{MyID:062}.
The seminal work \cite{CD:FredericoTorres:2007} makes use of the notion
of fractional Euler--Lagrange extremal introduced by
\cite{CD:Riewe:1996,CD:Riewe:1997} to prove a Noether-type theorem that
combines conservative and non-conservative cases. Another fractional
Noether-type theorem is found in \cite{Atanack}.
Fractional versions of Noether's theorem for isoperimetric problems
are the subject of the present work.

The text is organized in four sections. Section~\ref{sec:fdRL}
recalls the definitions from fractional calculus needed in the sequel
and fix the notations. Our results are formulated and proved
in Section~\ref{sec-ELRL}: we use a fractional operator
to generalize the classical concept of conservation law in mechanics
and we obtain a general fractional version of Noether's theorem
valid along the fractional isoperimetric
Euler--Lagrange extremals (Theorem~\ref{theo:TNfRL});
then we consider a more general fractional isoperimetric optimal control problem,
obtaining the corresponding fractional Noether's theorem in Hamiltonian form
(Theorem~\ref{thm:mainResult:FDA06}). Section~\ref{sub:sec:ex}
illustrates and discusses the new results with examples.

%--------------------------------------

\section{Preliminaries on Fractional Calculus}
\label{sec:fdRL}

In this section we fix notations by collecting the
necessary definitions of fractional derivatives
in the sense of Riemann--Liouville
\cite{Rev:02,CD:Hilfer:2000,Kilbas,CD:Podlubny:1999}.

\begin{definition}(Riemann--Liouville fractional integrals)
Let $f$ be defined on the interval $[a,b]$.
For $t \in [a,b]$, the left Riemann--Liouville fractional integral
$_aI_t^\alpha f$ and the right Riemann--Liouville fractional
integral $_tI_b^\alpha f$ of order $\alpha$, $\alpha >0$,
are defined by
\begin{gather*}
_aI_t^\alpha f(t)
=\frac{1}{\Gamma(\alpha)}\int_a^t (t-\theta)^{\alpha-1}f(\theta)d\theta \, ,\\
_tI_b^\alpha f(t) = \frac{1}{\Gamma(\alpha)}\int_t^b
(\theta-t)^{\alpha-1}f(\theta)d\theta \, ,
\end{gather*}
where $\Gamma$ is the Euler gamma function.
\end{definition}

\begin{definition}(Riemann--Liouville derivatives)
Let $f$ be defined on the interval $[a,b]$.
For $t \in [a,b]$, the left Riemann--Liouville fractional
derivative $_aD_t^\alpha f$ and the right Riemann--Liouville
fractional derivative $_tD_b^\alpha f$ of order $\alpha$ are defined by
\begin{equation}
\label{eq:DFRLE}
\begin{split}
_aD_t^\alpha f(t)&={D^n}_aI_t^{n-\alpha} f(t)\\
&=\frac{1}{\Gamma(n-\alpha)}\left(\frac{d}{dt}\right)^{n}
\int_a^t (t-\theta)^{n-\alpha-1}f(\theta)d\theta \, ,
\end{split}
\end{equation}
and
\begin{equation}
\label{eq:DFRLD}
\begin{split}
_tD_b^\alpha f(t) &= {(-D)^n}_tI_b^{n-\alpha}f(t)\\
&=\frac{1}{\Gamma(n-\alpha)}\left(-\frac{d}{dt}\right)^{n}
\int_t^b(\theta - t)^{n-\alpha-1}f(\theta)d\theta \, ,
\end{split}
\end{equation}
where $n \in \mathbb{N}$ is such that
$n-1 \leq \alpha < n$, and $D$ is the usual derivative.
\end{definition}

\begin{remark}
If $\alpha$ is an integer, then from \eqref{eq:DFRLE}
and \eqref{eq:DFRLD} one obtains the standard derivatives, that is,
\begin{equation*}
\label{eq:DU}
_aD_t^\alpha f(t) = \left(\frac{d}{dt}\right)^\alpha f(t) \, , \quad
_tD_b^\alpha f(t) = \left(-\frac{d}{dt}\right)^\alpha f(t) \, .
\end{equation*}
\end{remark}

\begin{theorem}
Let $f$ and $g$ be two continuous functions on $[a,b]$
and $p>0$. The following property holds
for all $t \in [a,b]$:
$_aD_t^p\left(f(t)+g(t)\right)
= {_aD_t^p}f(t)+{_aD_t^p}g(t)$.
\end{theorem}

\begin{remark}
In general, the Riemann--Liouville fractional
derivative of a constant $c$ is not equal to zero.
More precisely, one has
$$
_aD_t^\alpha(c) = \frac{c}{\Gamma(1-\alpha)} (t - a)^{-\alpha}.
$$
\end{remark}

\begin{remark}
The left Riemann--Liouville fractional derivative of order $p>0$
of function $(t-a)^\upsilon$, $\upsilon>-1$,
is given by
\begin{equation*}
_aD_t^p(t-a)^\upsilon
= \frac{\Gamma(\upsilon+1)}{\Gamma(-p+\upsilon+1)}(t-a)^{\upsilon-p} \, .
\end{equation*}
\end{remark}

The reader interested in additional background
on fractional calculus and more general fractional
operators is referred to \cite{Rev:01,Virginia94,Virginia08,MyID:226}.
For applications in physics see \cite{CD:Hilfer:2000}.

%--------------------------------------

\section{Main Results}
\label{sec-ELRL}

In \cite{CD:RiRuDe:2011} a formulation of the
Euler--Lagrange equations was given for isoperimetric problems
of the calculus of variations with fractional derivatives
in the sense of Riemann--Liouville.
Here we prove a fractional version of Noether's theorem
valid along the fractional isoperimetric
Euler--Lagrange extremals. For that we
introduce an appropriate fractional operator that allow us
to generalize the classical concept of conservation law.
Under the extended fractional notion of conservation law,
we begin by proving in \S\ref{sub:sec:CM} a fractional
Noether theorem without changing the time
variable $t$, \textrm{i.e.}, without transformation
of the independent variable (Theorem~\ref{theo:tnadf1}).
In \S\ref{sub:sec:NT} we proceed with a time-reparameterization
technique to obtain the fractional Noether's theorem in its general form
(Theorem~\ref{theo:TNfRL}). Finally, in \S\ref{sub:sec:OC}
we consider more general fractional isoperimetric optimal control problems,
obtaining the corresponding fractional Noether's theorem in Hamiltonian form
(Theorem~\ref{thm:mainResult:FDA06}).

%--------------------------------------

\subsection{On the fractional isoperimetric
Riemann--Liouville conservation of momentum}
\label{sub:sec:CM}

We begin by defining the fractional
isoperimetric problem under consideration.

\begin{problem}(The fractional isoperimetric problem)
\label{Pb1}
The fractional isoperimetric  problem
of the calculus of variations in the sense of Riemann--Liouville
consists to find the stationary functions of the functional
\begin{equation}
\label{Pf}
I[q(\cdot)] = \int_a^b L\left(t,q(t),{_aD_t^\alpha} q(t)\right) dt
\end{equation}
subject to $k \in \mathbb{N}$ isoperimetric equality constraints
\begin{equation}
\label{CT}
\int_a^b
g_j\left(t,q(t),{_aD_t^\alpha} q(t)\right) dt=l_j,\quad j=1,\ldots,k,
\end{equation}
and $2n$ boundary conditions
\begin{equation}
\label{bc}
q(a)=\phi\,,\quad q(b)=\psi,
\end{equation}
where $[a,b] \subset \mathbb{R}$, $a<b$,
$0 < \alpha< 1$, $l_j$, $j=1,\ldots,k$, are $k$
specified real constants, and the admissible functions
$q: t \mapsto q(t)$ and the Lagrangian
$L : (t,q,v_l) \mapsto L(t,q,v_l)$ are assumed to be
functions of class $C^2$:
\begin{gather*}
q(\cdot) \in C^2\left([a,b];\,\mathbb{R}^n \right),\\
L(\cdot,\cdot,\cdot) \in  C^2\left([a,b]\times\mathbb{R}^n
\times \mathbb{R}^n;\,\mathbb{R}\right).
\end{gather*}
\end{problem}

\begin{remark}
When $\alpha \rightarrow 1$,
Problem~\ref{Pb1} is reduced to
the classical isoperimetric problem
of the calculus of variations:
\begin{gather}
I[q(\cdot)] = \int_a^b L\left(t,q(t),\dot{q}(t)\right) dt
\longrightarrow \min, \label{isocv}\\
\int_a^b g_j\left(t,q(t),\dot{q}(t)\right) dt=l_j, \label{CT1}
\end{gather}
$j=1,\ldots,k$, subject to the boundary conditions \eqref{bc}.
For a modern account to isoperimetric variational problems
see \cite{MyID:136,MyID:125,MyID:206}.
\end{remark}

The arguments of the calculus of variations assert that by
using the Lagrange multiplier rule,
Problem~\ref{Pb1} is equivalent to the following
augmented problem \cite[$\S12.1$]{CD:Gel:1963}: to minimize
\begin{equation}
\label{agp}
\begin{split}
I[q(\cdot),\lambda]
&= \int_a^b F\left(t,q(t),{_aD_t^\alpha} q(t),\lambda\right)dt\\
&:=\int_a^b
\left[L\left(t,q(t),{_aD_t^\alpha} q(t)\right)
-\lambda \cdot g\left(t,q(t),{_aD_t^\alpha} q(t)\right)\right] dt
\end{split}
\end{equation}
subject to \eqref{bc}. The augmented Lagrangian
\begin{equation}
\label{eq:aug:Lag}
F:=L-\lambda \cdot g,
\end{equation}
$\lambda =\left(\lambda_1, \ldots,\lambda_k\right)\in\mathbb{R}^k$,
has an important role in our study.

The notion of extremizer (a local minimizer or a local maximizer)
to Problem~\ref{Pb1} is found in \cite{CD:RiRuDe:2011}. Extremizers
can be classified as normal or abnormal.

\begin{definition}
An extremizer of Problem~\ref{Pb1} that does not satisfy
\begin{equation}
\label{eq:frac:EL}
\partial_2 g\left(t,q(t),{_aD_t^\alpha} q(t)\right)
+ {_tD_b^\alpha}\partial_3 g\left(t,q(t),{_aD_t^\alpha} q(t)\right) = 0,
\end{equation}
where $\partial_i g$ denotes the partial derivative of
$g(\cdot,\cdot,\cdot)$ with respect to its $i$th argument,
is said to be a normal extremizer; otherwise
(i.e., if it satisfies \eqref{eq:frac:EL} for all $t\in[a,b]$),
is said to be abnormal.
\end{definition}

Next theorem summarizes the main result of \cite{CD:RiRuDe:2011}.

\begin{theorem}(see \cite{CD:RiRuDe:2011})
\label{Thm:FractELeq1}
If $q(\cdot)$ is a normal extremizer
to Problem~\ref{Pb1}, then it satisfies the
following \emph{fractional isoperimetric Euler--Lagrange
equation in the sense of Riemann--Liouville}:
\begin{equation}
\label{eq:eldf}
\partial_{2} F\left(t,q(t),{_aD_t^\alpha q(t)},\lambda\right)
+ {_tD_b^\alpha}\partial_{3} F\left(t,q(t),{_aD_t^\alpha q(t)},\lambda\right)  = 0,
\end{equation}
$t \in [a,b]$, where $F$ is the augmented Lagrangian \eqref{eq:aug:Lag}
associated with Problem~\ref{Pb1}.
\end{theorem}

\begin{remark}
When $\alpha \rightarrow 1$, the fractional isoperimetric Euler--Lagrange
equation \eqref{eq:eldf} is reduced to the classical
isoperimetric Euler--Lagrange equation
\begin{equation*}
\partial_{2} F\left(t,q(t),\dot{q}(t),\lambda\right)
-\frac{d}{dt}\partial_3 F\left(t,q(t),\dot{q}(t),\lambda\right)=0
\end{equation*}
(see, e.g., \cite[$\S4.2$]{CD:Bruce:2004}).
\end{remark}

Theorem~\ref{Thm:FractELeq1} leads to the concept of
isoperimetric fractional extremal in the sense of Riemann--Liouville.

\begin{definition}(Fractional isoperimetric extremal)
\label{def:seC}
A function $q(\cdot)$ that is a solution of \eqref{eq:eldf}
is said to be a \emph{fractional isoperimetric Riemann--Liouville extremal}
for Problem~\ref{Pb1}.
\end{definition}

In order to prove a fractional isoperimetric Noether's theorem,
we adopt a technique used in
\cite{CD:FredericoTorres:2007,CD:FredericoTorres:2010,CD:JMS:Torres:2002a}.
For that, we use \eqref{agp} to introduce the notion of variational
invariance and formulate a necessary condition
of invariance without transformation
of the independent variable $t$.

\begin{definition}(Invariance of \eqref{agp} without transforming $t$)
\label{def:inv1:MR}
Functional \eqref{agp} is invariant under an $\varepsilon$-parameter
group of infinitesimal transformations
$\bar{q}(t)= q(t) + \varepsilon\xi(t,q) + o(\varepsilon)$ if
\begin{equation}
\label{eq:invdf}
\int_{t_{a}}^{t_{b}} F\left(t,q(t),{_aD_t^\alpha q(t)},\lambda\right) dt
= \int_{t_{a}}^{t_{b}} F\left(t,\bar{q}(t),{_aD_t^\alpha
\bar{q}(t)},\lambda\right) dt
\end{equation}
for any subinterval $[{t_{a}},{t_{b}}] \subseteq [a,b]\,.$
\end{definition}

The next theorem establishes
a necessary condition of invariance.

\begin{theorem}(Necessary condition of invariance)
If functional \eqref{agp} is invariant in the sense of
Definition~\ref{def:inv1:MR}, then
\begin{equation}
\label{eq:cnsidf}
\partial_{2} F\left(t,q(t),{_aD_t^\alpha q(t)},\lambda\right) \cdot \xi(t,q(t))
+ \partial_{3} F\left(t,q(t),{_aD_t^\alpha q(t)},\lambda\right)
\cdot {_aD_t^\alpha \xi(t,q(t))} = 0.
\end{equation}
\end{theorem}

\begin{proof}
Having in mind that condition \eqref{eq:invdf} is valid for any
subinterval $[{t_{a}},{t_{b}}] \subseteq [a,b]$, we can get rid
off the integral signs in \eqref{eq:invdf}. Differentiating this
condition with respect to $\varepsilon$, then substituting
$\varepsilon=0$, and using the definitions and properties of the
fractional derivatives given in Section~\ref{sec:fdRL}, we arrive
to the intended conclusion:
\begin{multline}
\label{eq:SP}
0 = \partial_{2} (L-\lambda \cdot g)\left(t,q,{_aD_t^\alpha q}\right)\cdot\xi(t,q)
+ \partial_{3} (L-\lambda \cdot g)\left(t,q,{_aD_t^\alpha q}\right)\\
\times \frac{d}{d\varepsilon} \Biggl[\frac{1}{\Gamma(n-\alpha)}
\left(\frac{d}{dt}\right)^{n} \int_a^t (t-\theta)^{n-\alpha-1}q(\theta)d\theta \\
+\frac{\varepsilon}{\Gamma(n-\alpha)}\left(\frac{d}{dt}\right)^{n}\int_a^t
(t-\theta)^{n-\alpha-1}\xi(\theta,q)d\theta\Biggr]_{\varepsilon=0}.
\end{multline}
Expression \eqref{eq:SP} is equivalent to \eqref{eq:cnsidf}.
\end{proof}

The following definition is useful in order to introduce an
appropriate concept of \emph{fractional isoperimetric conservation law
in the sense of Riemann--Liouville}.

\begin{definition}(cf. Definition~19 of \cite{CD:FredericoTorres:2007})
\label{def:oprl}
Given two functions $f$ and $h$ of class
$C^1$ in the interval $[a,b]$, we introduce the following
operator:
\begin{equation*}
\mathcal{D}_{t}^{\gamma}\left(f,h\right) = -h \cdot {_tD_b^\gamma} f
+ f \cdot {_aD_t^\gamma} h \, ,
\end{equation*}
where $t \in [a,b]$ and $\gamma \in \mathbb{R}_0^+$.
\end{definition}

\begin{remark}
\label{rem:oprl}
In the classical context one has $\gamma=1$ and
$$
\mathcal{D}_{t}^{1}\left(f,h\right)=f' \cdot h + f \cdot h' \\
= \frac{d}{dt}(f \cdot h)=\mathcal{D}_{t}^{1}\left(h,f\right) \, .
$$
Roughly speaking, $\mathcal{D}_{t}^{\gamma}\left(f,h\right)$
is a fractional version of the derivative of the product
of $f$ with $h$. Differently from the classical context,
in the fractional case one has, in general,
$\mathcal{D}_{t}^{\gamma}\left(f,h\right)
\ne \mathcal{D}_{t}^{\gamma}\left(h,f\right)$.
We recall that the Leibniz formula, as we know it from standard calculus,
is not valid for fractional derivatives.
\end{remark}

We now prove the fractional isoperimetric Noether's theorem in the
sense of Riemann--Liouville without transformation
of the independent variable $t$.

\begin{theorem}(The Noether law of fractional momentum)
\label{theo:tnadf1}
If \eqref{agp} is invariant in the
sense of Definition~\ref{def:inv1:MR}, then
\begin{equation}
\label{eq:LC:Frac:RL}
\mathcal{D}_{t}^{\alpha}\left[\partial_{3}
F\left(t,q(t),{_aD_t^\alpha q(t)},\lambda\right),\xi(t,q(t))\right] = 0
\end{equation}
along any fractional isoperimetric Riemann--Liouville extremal $q(t)$, $t \in [a,b]$
(Definition~\ref{def:seC}).
\end{theorem}

\begin{proof}
We use the fractional Euler--Lagrange equations
\begin{equation*}
\partial_{2} (L-\lambda \cdot g)\left(t,q,{_aD_t^\alpha q}\right)
= -{_{t}D_b^\alpha}\partial_{3} (L-\lambda \cdot g)\left(t,q,{_aD_t^\alpha q}\right)
\end{equation*}
in \eqref{eq:cnsidf}, obtaining
\begin{equation*}
\begin{split}
0&=-{_{t}D_b^\alpha}\partial_{3}
(L-\lambda \cdot g)\left(t,q,{_aD_t^\alpha q}\right)\cdot\xi(t,q)
+\partial_{3} (L-\lambda \cdot g)\left(
t,q,{_aD_t^\alpha q}\right)\cdot{_aD_t^\alpha \xi(t,q)}\\
&= \mathcal{D}_{t}^{\alpha}\left(\partial_{3} (L-\lambda
\cdot g)\left(t,q,{_aD_t^\alpha q}\right),\xi(t,q)\right) \, .
\end{split}
\end{equation*}
The proof is complete.
\end{proof}

\begin{remark}
\label{rem:22}
When $\alpha  \rightarrow 1$, we obtain from
\eqref{eq:LC:Frac:RL} the following conservation law
applied to the isoperimetric problem \eqref{isocv}--\eqref{CT1}:
\begin{equation*}
\frac{d}{dt} \left[ \partial_{3}
F\left(t,q(t),\dot{q}(t),\lambda\right)\cdot\xi(t,q(t)) \right] = 0
\end{equation*}
along any isoperimetric Euler--Lagrange extremal $q(\cdot)$.
For this reason, we call to the fractional isoperimetric law \eqref{eq:LC:Frac:RL}
\emph{the fractional isoperimetric Riemann--Liouville conservation of momentum}.
\end{remark}

%--------------------------------------

\subsection{The fractional isoperimetric Noether theorem}
\label{sub:sec:NT}

The next definition gives a more general notion
of invariance for the integral functional \eqref{agp}.
The main result of this section,
Theorem~\ref{theo:TNfRL}, is formulated
with the help of this definition.

\begin{definition}(Invariance of \eqref{agp})
\label{def:invadf} The integral functional \eqref{agp}
is said to be invariant under the one-parameter group
of infinitesimal transformations
\begin{equation}
\label{eq:tinf2}
\begin{cases}
\bar{t} = t + \varepsilon\tau(t,q) + o(\varepsilon) \, ,\\
\bar{q}(t) = q(t) + \varepsilon\xi(t,q) + o(\varepsilon) \, ,\\
\end{cases}
\end{equation}
if
\begin{equation*}
\int_{t_{a}}^{t_{b}} F\left(t,q(t),{_aD_t^\alpha q(t)},\lambda\right) dt
= \int_{\bar{t}(t_a)}^{\bar{t}(t_b)}
F\left(\bar{t},\bar{q}(\bar{t}), {_aD_t^\alpha
\bar{q}(\bar{t})},\lambda\right) d\bar{t}
\end{equation*}
for any subinterval $[{t_{a}},{t_{b}}] \subseteq [a,b]$.
\end{definition}

Our next theorem gives a formulation of Noether's principle
to fractional isoperimetric problems of the calculus
of variations in the sense of Riemann--Liouville.

\begin{theorem}(Fractional isoperimetric Noether's theorem)
\label{theo:TNfRL}
If the integral functional \eqref{agp} is invariant
in the sense of Definition~\ref{def:invadf}, then
\begin{multline}
\label{eq:tndf}
\mathcal{D}_{t}^{\alpha}\left(F\left(t,q,{_aD_t^\alpha q},\lambda\right)
- \alpha\partial_{3} F\left(t,q,{_aD_t^\alpha q},\lambda\right)
\cdot{_aD_t^\alpha q}, \tau(t,q)\right)\\
+ \mathcal{D}_{t}^{\alpha}\left(\partial_{3}
F\left(t,q,{_aD_t^\alpha q},\lambda\right), \xi(t,q)\right) = 0
\end{multline}
along any fractional isoperimetric Riemann--Liouville extremal $q(\cdot)\,.$
\end{theorem}

\begin{proof}
We reparameterize the time (the independent variable $t$)
with a Lipschitzian transformation
$[\sigma_a,\sigma_b]\ni \sigma \mapsto t(\sigma)=\sigma f(\delta)\in [a,b]$
that satisfies
\begin{equation}
\label{eq:condla}
t_{\sigma}^{'} =\frac{dt(\sigma)}{d\sigma}=
f(\delta) = 1\,\, \text{ if } \,\, \delta=0\,.
\end{equation}
In this way one reduces \eqref{agp} to an autonomous
integral functional:
\begin{equation}
\label{eq:tempo}
\bar{I}[t(\cdot),q(t(\cdot)),\lambda]
=\int_{\sigma_{a}}^{\sigma_{b}}
(L-\lambda \cdot g)\left(t(\sigma),q(t(\sigma)),
{_{\sigma_{a}}D_{t(\sigma)}^{\alpha}q(t(\sigma))}
\right)t_{\sigma}^{'} d\sigma ,
\end{equation}
where $t(\sigma_{a}) = a$, $t(\sigma_{b}) = b$,
\begin{equation*}
\begin{split}
_{\sigma_{a}}D_{t(\sigma)}^{\alpha}q(t(\sigma))
&=\frac{1}{\Gamma(n-\alpha)}\left(\frac{d}{dt(\sigma)}\right)^{n}
\int_{\frac{a}{f(\delta)}}^{\sigma f(\delta)}\left({\sigma
f(\delta)}-\theta\right)^{n-\alpha-1}q\left(\theta f^{-1}(\delta)\right)d\theta\\
&=\frac{(t_{\sigma}^{'})^{-\alpha}}{\Gamma(n-\alpha)}
\left(\frac{d}{d\sigma}\right)^{n}
\int_{\frac{a}{(t_{\sigma}^{'})^{2}}}^{\sigma}
(\sigma-s)^{n-\alpha-1}q(s)ds  \\
&=(t_{\sigma}^{'})^{-\alpha}{_{\frac{a}{(t_{\sigma}^{'})^{2}}}
D_{\sigma}^{\alpha}q(\sigma)}.
\end{split}
\end{equation*}
Using the definitions and properties of fractional derivatives given in
Section~\ref{sec:fdRL}, we get
\begin{equation*}
\begin{split}
\bar{I}[t(\cdot),q(t(\cdot)),\lambda]
&= \int_{\sigma_{a}}^{\sigma_{b}}
(L-\lambda \cdot g)\left(t(\sigma),q(t(\sigma)),
(t_{\sigma}^{'})^{-\alpha}{_{\frac{a}{(t_{\sigma}^{'})^{2}}}
D_{\sigma}^{\alpha}}q(\sigma)\right) t_{\sigma}^{'} d\sigma \\
&= \int_{\sigma_{a}}^{\sigma_{b}}
\left(\bar{L}_{f}-\lambda\cdot \bar{g}_{f}\right)\left(t(\sigma),
q(t(\sigma)),t_{\sigma}^{'},{_{\frac{a}{(t_{\sigma}^{'})^{2}}}
D_\sigma^\alpha} q(t(\sigma))\right)d\sigma \\
&= \int_a^b (L-\lambda \cdot g)\left(t,q(t),{_aD_t^\alpha} q(t)\right) dt\\
&= I[q(\cdot),\lambda] \, .
\end{split}
\end{equation*}
By hypothesis, functional \eqref{eq:tempo} is invariant
under transformations \eqref{eq:tinf2},
and it follows from Theorem~\ref{theo:tnadf1} that
if the integral functionals in \eqref{Pf} and \eqref{CT}
are invariant in the sense of Definition~\ref{def:invadf},
then the integral functional \eqref{eq:tempo} is invariant
in the sense of Definition~\ref{def:inv1:MR}. It follows
from Theorem~\ref{theo:tnadf1} that
\begin{equation}
\label{eq:tnadf2}
\mathcal{D}_{t}^{\alpha}\left(\partial_{4}\left(\bar{L}_{f}
-\lambda \cdot\bar{g}_{f}\right),\xi\right)
+ \mathcal{D}_{t}^{\alpha}\left(\frac{\partial}{\partial t'_\sigma}
\left(\bar{L}_{f}-\lambda \cdot\bar{g}_{f}\right),\tau\right) = 0
\end{equation}
is an isoperimetric fractional conserved law in the sense of Riemann--Liouville.
For $\delta= 0$ the condition \eqref{eq:condla} allow us to write that
\begin{equation*}
_{\frac{a}{(t_{\sigma}^{'})^{2}}}D_\sigma^\alpha q(t(\sigma))
 = {_aD_t}^\alpha q(t) \, ,
\end{equation*}
and we get
\begin{equation}
\label{eq:prfMR:q1}
\partial_{4}\left(\bar{L}_{f}-\lambda \cdot\bar{g}_{f}\right)
=\partial_{3} \left(L-\lambda \cdot g\right)  \, ,
\end{equation}
and
\begin{equation}
\label{eq:prfMR:q2}
\frac{\partial}{\partial t'_\sigma} \left(\bar{L}_{f}-\lambda\cdot \bar{g}_{f}\right)
= -\alpha\partial_{3} (L-\lambda \cdot g)\cdot{_{a}D_{t}^{\alpha}}q + L-\lambda \cdot g \, .
\end{equation}
Substituting the quantities \eqref{eq:prfMR:q1} and
\eqref{eq:prfMR:q2} into \eqref{eq:tnadf2}, we obtain the
isoperimetric fractional conservation law \eqref{eq:tndf}.
\end{proof}

\begin{remark}
\label{rem:25}
When $\alpha \rightarrow 1$, we obtain from \eqref{eq:tndf}
the isoperimetric Noether's conservation law:
\begin{equation*}
\frac{d}{dt}
\biggl[\partial_{3} F\left(t,q,\dot{q}\right)\cdot\xi(t,q)
+ \left( F(t,q,\dot{q}) - \partial_{3} F\left(t,q,\dot{q}\right)
\cdot \dot{q} \right) \tau(t,q)\biggr] = 0
\end{equation*}
along any Euler--Lagrange extremal $q$
of problem \eqref{isocv}--\eqref{CT1}.
\end{remark}

%--------------------------------------

\subsection{Optimal control of fractional isoperimetric systems}
\label{sub:sec:OC}

We now adopt the Hamiltonian formalism to generalize
Theorem~\ref{theo:TNfRL} to the fractional optimal
control setting. The fractional isoperimetric
optimal control problem in the sense of Riemann--Liouville
is introduced, without loss of generality, in Lagrange form:
\begin{equation}
\label{eq:COA} I[q(\cdot),u(\cdot)] = \int_a^b
L\left(t,q(t),u(t)\right) dt \longrightarrow \min
\end{equation}
subject to the fractional differential system
\begin{equation}
\label{eq:sitRL}
 _aD_t^\alpha
q(t)=\varphi\left(t,q(t),u(t)\right),
\end{equation}
isoperimetric equality constraints
\begin{equation}
\label{CT2}
\int_a^b
g_j\left (t,q(t),u(t)\right) dt=l_j,\quad j=1,\ldots,k,
\end{equation}
and initial condition
\begin{equation}
\label{eq:COIRL}
q(a)=q_a.
\end{equation}

The Lagrangian $L :[a,b] \times \mathbb{R}^{n}\times
\mathbb{R}^{m} \rightarrow \mathbb{R}$, the fractional velocity vector
$\varphi:[a,b] \times \mathbb{R}^{n}\times \mathbb{R}^m\rightarrow
\mathbb{R}^{n}$ and $g : [a,b] \times \mathbb{R}^{n}\times \mathbb{R}^m
\rightarrow \mathbb{R}^{k}$, are assumed to be functions of class $C^{1}$
with respect to all their arguments, and $l_j$, $j=1,\ldots,k$,
are specified real constants. We also assume, without loss of
generality, that $0<\alpha\leq1$. In conformity with the calculus
of variations, we are considering that the control functions
$u(\cdot)$ take values on $\mathbb{R}^m$.

\begin{definition}
The fractional differential system
\eqref{eq:sitRL} is called a
\emph{fractional control system
in the sense of Riemann--Liouville}.
\end{definition}

\begin{remark}
\label{rem:cv:pc}
The fractional functional of the calculus of variations
\eqref{Pf} is obtained from \eqref{eq:COA}--\eqref{eq:sitRL}
by choosing $\varphi(t,q,u)=u$. In that case \eqref{CT2}
is reduced to \eqref{CT}.
\end{remark}

\begin{definition}(Fractional isoperimetric process)
An admissible pair $(q(\cdot),u(\cdot))$ that satisfies the
fractional control system \eqref{eq:sitRL} and the
fractional isoperimetric constraints \eqref{CT2}
is said to be a \emph{fractional isoperimetric process
in the sense of Riemann--Liouville}.
\end{definition}

\begin{theorem}
\label{th:AG}
If $(q(\cdot),u(\cdot))$ is a fractional isoperimetric process
in the sense of Riemann--Liouville,
solution to problem \eqref{eq:COA}--\eqref{eq:COIRL},
then there exists a co-vector function $p(\cdot)\in PC^{1}([a,b];\mathbb{R}^{n})$
such that for all $t\in [a,b]$ the
quadruple $(q(\cdot),u(\cdot),p(\cdot),\lambda)$ satisfies
the following conditions:
\begin{itemize}
\item the isoperimetric Hamiltonian system
\begin{equation*}
\label{eq:HamRL}
\begin{cases}
_aD_t^\alpha q(t) =\partial_4 {\cal H}(t, q(t), u(t),p(t),\lambda) \, , \\
_tD_b^\alpha p(t) = \partial_2{\cal H}(t,q(t),u(t), p(t),\lambda) \, ;
\end{cases}
\end{equation*}
\item the isoperimetric stationary condition
\begin{equation*}
 \partial_3 {\cal H}(t, q(t), u(t), p(t),\lambda)=0 \, ;
\end{equation*}
\end{itemize}
where the Hamiltonian ${\cal H}$ is defined  by
\begin{equation}
\label{eq:HL}
{\cal H}\left(t,q,u,p,\lambda\right)
= L\left(t,q,u\right)-\lambda \cdot g\left(t,q,u\right)
+ p \cdot \varphi\left(t,q,u\right)\,.
\end{equation}
\end{theorem}

\begin{proof}
Minimizing \eqref{eq:COA} subject to
\eqref{eq:sitRL} and \eqref{CT2} is equivalent,
by the Lagrange multiplier rule,
to minimize
\begin{equation}
\label{eq:COA1}
J[q(\cdot),u(\cdot),p(\cdot),\lambda]
= \int_a^b \left[{\cal H}\left(t,q(t),u(t),p(t),\lambda\right)
-p(t) \cdot {_aD}_t^\alpha q(t)\right]dt
\end{equation}
with ${\cal H}$ given by \eqref{eq:HL}.
Theorem~\ref{th:AG} follows by applying the
fractional Euler--Lagrange optimality condition
to the equivalent functional \eqref{eq:COA1}.
\end{proof}

\begin{remark}
When $\alpha \rightarrow 1$, Theorem~\ref{th:AG}
coincides with the Pontryagin Maximum Principle for
optimal control problems with isoperimetric constraints
(cf. \cite[$\S13.12$]{MR84m:49002} and \cite[Theorem $2.1$]{MR1901565}).
\end{remark}

\begin{remark}
In the case of the fractional calculus
of variations in the sense of Riemann--Liouville
one has $\varphi(t,q,u)=u$ (Remark~\ref{rem:cv:pc})
and ${\cal H} = L-\lambda \cdot g + p \cdot u$.
From the isoperimetric Hamiltonian system of Theorem~\ref{th:AG},
one gets $_aD_t^\alpha q = u$ and
$_tD_b^\alpha  p  =\partial_2 L-\lambda\cdot\partial_2 g$,
and from the stationary condition
$\partial_3 {\cal H} = 0$ it follows that
$p= - \partial_3 L +\lambda \cdot\partial_3g$. Thus,
${_tD_b^\alpha}  p= -_tD_b^\alpha \left(\partial_3 L
-\lambda\cdot\partial_3g\right)$.
Comparing both expressions for $_tD_b^\alpha  p$, we arrive to
the fractional Euler--Lagrange equations \eqref{eq:eldf}:
$\partial_2 L-\lambda\cdot\partial_2 g
=-_tD_b^\alpha \left(\partial_3 L-\lambda\cdot\partial_3g\right)$.
\end{remark}

\begin{definition}(Fractional isoperimetric Pontryagin extremal)
\label{def:extPontRL}
A quadruple $(q(\cdot),u(\cdot),p(\cdot),\lambda)$
satisfying Theorem~\ref{th:AG} will be called a \emph{fractional
isoperimetric Pontryagin extremal in the sense of Riemann--Liouville}.
\end{definition}

The notion of variational invariance for \eqref{eq:COA}--\eqref{CT2}
is defined with the help of the augmented functional \eqref{eq:COA1}.

\begin{definition}(Variational invariance of \eqref{eq:COA1})
\label{def:inv:gt1} We say that the
integral functional \eqref{eq:COA1}
is invariant under the one-parameter
family of infinitesimal transformations
\begin{equation}
\label{eq:trf:inf}
\begin{cases}
\bar{t} = t+\varepsilon\tau(t, q(t), u(t), p(t)) + o(\varepsilon) \, , \\
\bar{q}(t) = q(t)+\varepsilon\xi(t, q(t), u(t), p(t)) + o(\varepsilon) \, , \\
\bar{u}(t) = u(t)+\varepsilon\varrho(t, q(t), u(t), p(t)) + o(\varepsilon) \, , \\
\bar{p}(t) = p(t)+\varepsilon\varsigma(t, q(t), u(t), p(t))+ o(\varepsilon) \, , \\
\end{cases}
\end{equation}
if
\begin{equation}
\label{eq:condInv}
\left[{\cal H}(\bar{t},\bar{q}(\bar{t}),\bar{u}(\bar{t}),\bar{p}(\bar{t}),\lambda)
-\bar{p}(\bar{t}) \cdot  {_{\bar{a}}D_{\bar{t}}}^\alpha
\bar{q}(\bar{t})\right] d\bar{t}
=\left[{\cal H}(t,q(t),u(t),p(t),\lambda)-p(t)
\cdot {_aD_t^\alpha} q(t)\right] dt \, .
\end{equation}
\end{definition}

The next result provides an extension of Noether's
theorem to the wider class of fractional isoperimetric
optimal control problems.

\begin{theorem}(Noether's theorem in Hamiltonian form)
\label{thm:mainResult:FDA06}
If \eqref{eq:COA1} is variationally invariant,
in the sense of Definition~\ref{def:inv:gt1}, then
\begin{equation}
\label{eq:tndf:CO}
\mathcal{D}_{t}^{\alpha}\biggl({\cal H}(t,q(t),u(t),p(t),\lambda)
- \left(1 - \alpha\right) p(t) \cdot {_aD_t^\alpha} q(t),
\tau(t,q(t))\biggr)
- \mathcal{D}_{t}^{\alpha}\left(p(t), \xi(t,q(t))\right) = 0
\end{equation}
along any fractional isoperimetric Pontryagin extremal
$(q(\cdot),u(\cdot),p(\cdot),\lambda)$ of problem
\eqref{eq:COA}--\eqref{eq:COIRL}.
\end{theorem}

\begin{proof}
The fractional isoperimetric conservation law
\eqref{eq:tndf:CO} in the sense of Riemann--Liouville
is obtained by applying Theorem~\ref{theo:TNfRL}
to the equivalent functional \eqref{eq:COA1}.
\end{proof}

\begin{remark}
When $\alpha \rightarrow 1$, one gets from Theorem~\ref{thm:mainResult:FDA06}
the Noether-type theorem associated with the classical
isoperimetric optimal control problem \cite[Theorem $4.1$]{MR1901565}:
invariance under a one-parameter family of infinitesimal
transformations \eqref{eq:trf:inf} implies that
\begin{equation*}
{\cal H}(t,q(t),u(t),p(t),\lambda)\tau(t,q(t))-p(t)\cdot
\xi(t,q(t)) = constant
\end{equation*}
along all the Pontryagin extremals.
\end{remark}

%--------------------------------------

\section{Examples}
\label{sub:sec:ex}

We illustrate our results with the help of two
fractional isoperimetric problems. Example~\ref{ex:2}
considers a nonautonomous fractional isoperimetric problem
of the calculus of variations; Example~\ref{cor:FOCP:CO}
the autonomous optimal control isoperimetric problem.

\begin{example}
\label{ex:2}
Let $\alpha$ be a given number in the interval $(0,1)$.
Consider the following fractional isoperimetric problem:
\begin{equation*}
\begin{gathered}
\int_0^1(t^4+({_0D_t^\alpha}y)^2)dt \longrightarrow \min,\\
 \int_0^1 t^2{_0D_t^\alpha}y \,dt=\frac{1}{5},\\
y(0) = 0 \, , \quad y(1) = \frac{2}{2\alpha+3\alpha^2+\alpha^3} \, .
\end{gathered}
\end{equation*}
The augmented Lagrangian is
\begin{equation}
\label{ex11}
F(t,y,{_0D_t^\alpha}y)
=t^4+({_0D_t^\alpha}y )^2-\lambda \, t^2 {_0D_t^\alpha}y
\end{equation}
and in \cite{CD:RiRuDe:2011} it is proved that
\begin{equation}
\label{sol:ex:fip}
y(t) = \frac{1}{\Gamma(\alpha)}\int_0^t\frac{x^2}{(t-x)^{1-\alpha}}dx
= \displaystyle\frac{1}{\Gamma(\alpha)} \,
\frac{2t^{\alpha+2}}{2\alpha+3\alpha^2+\alpha^3}
\end{equation}
is an extremal if $\lambda=2$ and
\begin{equation}
\label{ex1}
{_0D_t^\alpha}y=t^2.
\end{equation}
It is easy to check the validity of our Theorem~\ref{theo:TNfRL} for this problem:
take $\xi = 1$, $\tau = 1$, and use \eqref{ex11}--\eqref{sol:ex:fip}--\eqref{ex1}
in \eqref{eq:tndf} to obtain ${\mathcal{D}_t^\alpha}(0,1)=0$.
\end{example}

Theorem~\ref{thm:mainResult:FDA06} gives an
interesting result for autonomous fractional problems.

\begin{example}
\label{cor:FOCP:CO}
Consider the autonomous fractional isoperimetric
optimal control problem, \textrm{i.e.}, the case when functions $L$,
$\varphi$ and $g$ of \eqref{eq:COA}--\eqref{CT2}
do not depend explicitly on the independent variable:
\begin{gather}
I[q(\cdot),u(\cdot)] =\int_a^b L\left(q(t),u(t)\right) dt
\longrightarrow \min \, , \label{eq:FOCP:CO}\\
_aD_t^\alpha
q(t)=\varphi\left(q(t),u(t)\right)\label{eq:FOCP:CO10}\,,\\
\int_a^bg_{j}(q(t),u(t))=l_j\,. \label{eq:45}
\end{gather}
We will show that for the fractional problem
\eqref{eq:FOCP:CO}--\eqref{eq:45} one has
\begin{equation}
\label{eq:frac:eng}
_aD_t^\alpha \left[
{\cal H}(t,q(t),u(t),p(t),\lambda) +
\left(\alpha-1\right) p(t) \cdot {_aD_t^\alpha} q(t)\right] = 0
\end{equation}
along any isoperimetric fractional Pontryagin extremal
$(q(\cdot),u(\cdot),p(\cdot),\lambda)$.
Indeed, as the Hamiltonian ${\cal H}$ does not depend explicitly on the
independent variable $t$, we can easily see that
\eqref{eq:FOCP:CO}--\eqref{eq:45} is invariant
under translation of the time variable:
the condition of invariance \eqref{eq:condInv} is satisfied with
$\bar{t}(t) = t+\varepsilon$,
$\bar{q}(t) = q(t)$,
$\bar{u}(t) = u(t)$,
and $\bar{p}(t) = p(t)$. Indeed,
given that $d\bar{t} = dt$, the invariance
condition \eqref{eq:condInv} is verified if
${_{\bar{a}}D_{\bar{t}}^\alpha} \bar{q}(\bar{t}) =
{_aD_t^\alpha} q(t)$. This is true because
\begin{equation*}
\begin{split}
_{\bar{a}} D_{\bar{t}}^\alpha \bar{q}(\bar{t})
&= \frac{1}{\Gamma(n-\alpha)}\left(\frac{d}{d\bar{t}}\right)^{n}
\int_{\bar{a}}^{\bar{t}} (\bar{t}-\theta)^{n-\alpha-1}\bar{q}(\theta)d\theta \\
&= \frac{1}{\Gamma(n-\alpha)}\left(\frac{d}{dt}\right)^{n}
\int_{a + \varepsilon}^{t+\varepsilon} (t + \varepsilon-\theta)^{n-\alpha-1}\bar{q}(\theta)d\theta \\
&= \frac{1}{\Gamma(n-\alpha)}\left(\frac{d}{dt}\right)^{n}
\int_{a}^{t} (t-s)^{n-\alpha-1}\bar{q}(s + \varepsilon)ds \\
&= {_{a}D_{t}}^\alpha \bar{q}(t + \varepsilon) = {_{a}D_{t}}^\alpha \bar{q}(\bar{t}) \\
&= {_{a}D_{t}}^\alpha q(t) \, .
\end{split}
\end{equation*}
Using the notation in \eqref{eq:trf:inf}, we have
$\tau = 1$, $\xi=\varrho=\varsigma=0$.
From Theorem~\ref{thm:mainResult:FDA06} we arrive to
the intended equality \eqref{eq:frac:eng}.
\end{example}

The Example~\ref{cor:FOCP:CO} shows that in contrast with
the classical autonomous isoperimetric problem of optimal control,
for \eqref{eq:FOCP:CO}--\eqref{eq:45}
the Hamiltonian ${\cal H}$ does not define a conservation law.
Instead of the classical equality
$\frac{d}{dt}\left({\cal H}\right)=0$, we have
\begin{equation}
\label{eq:ConsHam:alpha}
_aD_t^\alpha \left[
{\cal H} + \left(\alpha-1\right) p(t)
\cdot {_aD_t^\alpha} q(t) \right] = 0 \, ,
\end{equation}
\textrm{i.e.}, fractional conservation
of the Hamiltonian ${\cal H}$ plus a quantity that depends on
the fractional order $\alpha$ of differentiation.
This seems to be explained by violation of the homogeneity
of space-time caused by the fractional
derivatives, when $\alpha\neq 1$. If $\alpha=1$, then
we obtain from \eqref{eq:ConsHam:alpha} the classical result:
the Hamiltonian ${\cal H}$ is preserved along all the
isoperimetric Pontryagin extremals.

%--------------------------------------

\section*{Acknowledgments}

This work was supported by {\it FEDER} funds through
{\it COMPETE} --- Operational Programme Factors of Competitiveness
(``Programa Operacional Factores de Competitividade'')
and by Portuguese funds through the
{\it Center for Research and Development
in Mathematics and Applications} (University of Aveiro)
and the Portuguese Foundation for Science and Technology
(``FCT --- Funda\c{c}\~{a}o para a Ci\^{e}ncia e a Tecnologia''),
within project PEst-C/MAT/UI4106/2011
with COMPETE number FCOMP-01-0124-FEDER-022690.
Partially presented at The 5th Symposium
on Fractional Differentiation and its Applications (FDA'12),
held in Hohai University, Nanjing, May 14-17, 2012.
Frederico was also supported by the FCT post-doc
fellowship SFRH/BPD/51455/2011,
from the program {\it Ci\^{e}ncia Global};
Torres by FCT through the project PTDC/MAT/113470/2009
and by EU funding under the
7th Framework Programme FP7-PEOPLE-2010-ITN,
grant agreement number 264735-SADCO.
The authors are grateful to two anonymous referees
for valuable comments and suggestions.

%--------------------------------------

%--------------------------------------

\end{document}